\documentclass{amsart}

\usepackage{amsmath,amsfonts,amssymb,amsthm,pb-diagram,picinpar,graphicx,color}
\usepackage{slashbox}
\usepackage{colortbl}

\newcommand{\ga}{\alpha}

\newcommand{\gd}{\delta}
\newcommand{\gh}{\theta}
\newcommand{\gk}{\kappa}
\newcommand{\gl}{\lambda}
\newcommand{\gs}{\sigma}

\newcommand{\gC}{\Gamma}
\newcommand{\gD}{\Delta}

\newcommand{\gO}{\Omega}

\newcommand{\bA}{\mathbb{A}}
\newcommand{\bC}{\mathbb{C}}

\newcommand{\bL}{\mathbb{L}}
\newcommand{\bP}{\mathbb{P}}
\newcommand{\bV}{\mathbb{V}}

\newcommand{\fm}{\mathfrak{m}}

\newcommand{\fA}{\mathfrak{A}}

\newcommand{\fS}{\mathfrak{S}}

\newcommand{\dP}{\check{\mathbb{P}}}

\newcommand{\calA}{\mathcal{A}}

\newcommand{\calE}{\mathcal{E}}

\newcommand{\calL}{\mathcal{L}}
\newcommand{\calM}{\mathcal{M}}

\newcommand{\calO}{\mathcal{O}}

\newcommand{\calT}{\mathcal{T}}

\newcommand{\wtX}{\widetilde{X}}

\newtheorem{Th}{Theorem}[section]
\newtheorem{Prop}[Th]{Proposition}
\newtheorem{Lem}[Th]{Lemma}
\newtheorem{Cor}[Th]{Corollary}
\newtheorem{Rem}[Th]{Remark}

\newtheorem{Problem}[Th]{Problem}

\makeatletter
 
 \@addtoreset{figure}{section}
\makeatother

\begin{document}

\title[A note on normal triple cover over $\bP^2$]{A note on normal triple covers over $\bP^2$ with branch divisors of degree $6$}
\author{Taketo Shirane}
\address{Department of Mathematics and Information Sciences, Tokyo Metropolitan University, 
1-1 Minamiohsawa, Hachioji 192-0397, Tokyo Japan}
\email{sirane-taketo@tmu.ac.jp}
\keywords{triple cover, cubic surface, torus curve}
\subjclass[2010]{14E20, 14E22}
\begin{abstract}
Let $S$ and $T$ be reduced divisors on $\bP^2$ 
which have no common components, 
and $\gD=S+2\,T.$ 
We assume $\deg\gD=6.$ 
Let $\pi:X\to\bP^2$ be a normal triple cover with branch divisor $\gD,$ 
i.e. $\pi$ is ramified along $S$ (resp. $T$) with the index $2$ (resp. $3$).  
In this note, we show that $X$ is either a $\bP^1$-bundle over an elliptic curve or a normal cubic surface in $\bP^3.$ 
Consequently, we give a necessary and sufficient condition for $\gD$
to be the branch divisor of a normal triple cover over $\bP^2.$ 
\end{abstract}

\maketitle

\section*{Introduction}
The first systematic study on triple covers was done by Miranda \cite{miranda}. 
Afterwards, triple covers are studied by many mathematicians 
(e.g. \cite{casnati,fujita,tan,tokunaga0}). 
Yet it is difficult to deal with general triple covers. 
For example, the following fundamental problem still remains as an open problem. 

\begin{Problem}\label{prob 1}\rm
Let $\gD=S+2\,T$ be a divisor on $\bP^2=\bP^2_{\bC},$ 
where $S$ and $T$ are reduced divisors which have no common components. 
Give a necessary and sufficient condition for $\gD$ 
to be the branch divisor of a normal triple cover over $\bP^2$ (see below for the notation).  
\end{Problem}

The above problem is an analogy to \cite[Question~0.1]{ishida}. 
The difference between Problem~\ref{prob 1} and \cite[Question~0.1]{ishida} is whether a condition of ramification is given, or not. 
In some cases, Problem~\ref{prob 1} was solved by some mathematicians, mainly Tokunaga, as follows: 

If $S=0,$ then a normal triple cover $\pi:X\to\bP^2$ with branch divisor $\gD$ must be a Galois cover, 
hence one can see an answer of Problem~\ref{prob 1} from \cite{miranda}. 
In the cases where $(\deg S,\deg T)=(2,1),\ (2,2),\ (4,0)$ and $(4,1),$ 
Tokunaga solved Problem~\ref{prob 1} by using his theory of dihedral covers in \cite{tokunaga1} and \cite{tokunaga2}. 
Moreover, Yasumura showed that, 
if $\pi:X\to\bP^2$ is a normal triple cover with branch divisor $\gD$ and $(\deg S, \deg T)=(4,1),$ 
then $X$ is a cubic surface in $\bP^3$ and $\pi$ is a projection centered at a point of $\bP^3\setminus X$ (\cite{yasumura}). 
In the case where $T=0$ and $S$ is a sextic curve with at most simple singularities, 
Ishida and Tokunaga showed that 
$X$ is either a quotient of an abelian surface by an involution 
or a normal cubic surface in $\bP^3,$ 
and gave an answer of Problem~\ref{prob 1} (\cite{ishida}). 

The author is inspired by these results to do this study.  
The aim is to characterize normal triple covers over $\bP^2$ with branch divisors of degree $6,$ 
and to give an answer of Problem~\ref{prob 1} in the case $\deg\gD=6$ without any assumptions.

\paragraph{\bf Notation.}
The base field is the field of complex numbers $\bC$ throughout this note. 
We call a finite flat morphism $\pi:X\to Y$ from a scheme $X$ to a variety $Y$ a \textit{cover}. 
If, in addition, $X$ and $Y$ are normal varieties, 
we call $\pi$ a \textit{normal cover}. 
If the degree of a cover (resp. a normal cover) is three, we call it a \textit{triple cover} (resp. a \textit{normal triple cover}). 
Let $\pi:X\to Y$ be a normal triple cover over a non-singular variety $Y.$ 
We denote the branch locus in $Y$ of $\pi$ by $\gD_{\pi}.$ 
Then $\gD_{\pi}$ has purely codimension $1$ in $Y.$ 
Hence we can regard $\gD_{\pi}$ as a reduced divisor. 
Moreover we can decompose $\gD_{\pi}$ into $S_{\pi}+T_{\pi},$ 
where $\pi$ is ramified along $S_{\pi}$ (resp. $T_{\pi}$) with the index $2$ (resp. $3$). 
We denote $S_{\pi}+2\,T_{\pi}$ by $\overline{\gD}_{\pi}$ and call it the \textit{branch divisor} of $\pi.$ 
We say that $P\in Y$ is a \textit{total branched point} of $\pi$ if $\pi^{-1}(P)$ consists of one point. 

\begin{Rem}\rm
Since normal singularities of surfaces are Cohen-Macaulay, 
a finite surjective morphism from a normal surface to a smooth surface is a normal cover (cf. \cite{matsumura}). 
In \cite{artal, ishida, tokunaga0, tokunaga1, tokunaga3, tokunaga2, yasumura}, 
a normal triple cover over a smooth surface are simply called a ``triple cover''. 

\end{Rem}

\paragraph{\bf Main theorem.}
To state the main theorem, we introduce some notation. 
We denote the dual space of $\bP^2$ by $\dP^2.$ 
Let $F$ be the flag variety of pairs of points and lines in $\bP^2,$ 
and $p:F\to\bP^2$ and $q:F\to\dP^2$ the canonical projections. 
For a irreducible curve $\gC\subset\bP^2,$ we denote the dual curve of $\gC$ in $\dP^2$ by $\gC^{\vee}.$ 
We will show the following theorem based on Miranda's theory. 

\begin{Th}\label{th main}
Let $\pi:X\to\bP^2$ be a normal triple cover with $\deg\overline{\gD}_{\pi}=6.$ 
Then $\pi:X\to\bP^2$ satisfies either the following two conditions; 
\begin{enumerate}
\item[\rm (i)] 
if $S_{\pi}$ is a sextic curve with $9$ cusps (hence $\gD_{\pi}=S_{\pi}$), and 
the $9$ cusps are total branched points of $\pi,$ then 
$X\cong q^{-1}(\gD_{\pi}^{\vee})\subset F,$ and 
$\pi$ is the restriction of $p$ to $X$; or

\item[\rm (ii)] 
if $\pi$ does not satisfy the assumptions of {\rm (i)}, 
then $X$ is a cubic surface in $\bP^3,$ and
$\pi$ is a projection centered at a point of $\bP^3\setminus X.$ 
\end{enumerate} 
\end{Th}

\begin{Rem}\rm 
Let $\pi:X\to\bP^2$ be a normal triple cover satisfying the condition (i) in the above theorem. 
Let $\tilde{\pi}:\wtX\to\bP^2$ be the $\bL$-normalization of $\bP^2,$ 
where $\bL$ is the Galois closure of the extension of the rational function fields $\bC(X)/\bC(\bP^2).$ 
Then it is easy to see that $\wtX$ is isomorphic to $\gD_{\pi}^{\vee}\times\gD_{\pi}^{\vee}$ (cf. \cite{shirane3}), 
and $\tilde{\pi}$ is the $\fS_3$-cover in \cite[Example~6.3]{tokunaga3}. 
Moreover similar covers to $\pi$ are used to construct families of Galois closure curves in \cite{shirane2}. 
\end{Rem}

\begin{Rem}\rm
Let $\pi:X\to\bP^2$ be a normal triple cover with $\deg\overline{\gD}_{\pi}=6.$ 
Ishida and Tokunaga showed that, 
if $\gD_{\pi}$ is a sextic curve with at most simple singularities, 
then either $X$ is a quotient of an abelian surface by an involution or 
$\pi$ satisfies the condition (ii) in the above theorem (\cite{ishida}). 
Yasumura showed that, 
if $(\deg S_{\pi},\deg T_{\pi})=(4,1),$ 
then $\pi$ satisfies the condition (ii) in the above (\cite{yasumura}). 
(In the case where $(\deg S_{\pi},\deg T_{\pi})=(2,2),$ 
no characterization of $\pi$ was known.) 
Theorem~\ref{th main} is a generalization of these results without any assumptions. 
\end{Rem}

Consequently, we will show the following corollary, 
which is a generalization of \cite[Theorem~0.1]{ishida}. 

\begin{Cor}\label{cor 1}
Let $\gD$ be a divisor of degree $6$ on $\bP^2.$ 
Then there is a normal triple cover $\pi$ with $\overline{\gD}_{\pi}=\gD$ 
if and only if 
there are homogeneous polynomials $G_i(x_0,x_1,x_2)$ 
of degree $i$ for $i=1,2$ with the following three conditions:  
\begin{enumerate}
\item[\rm (1)] $G_2^3+G_3^2=0$ defines $\gD$; 
\item[\rm (2)] $G_2\not\in \fm_E$ or $G_3\not\in \fm_E^2$ for any prime divisor $E,$ 
where $\fm_E$ is the maximal ideal of the local ring $\calO_E$ at $E$; and 
\item[\rm (3)] $G_2\in \fm_E$ or $G_2^3+G_3^2\not\in \fm_E^2$ for any prime divisor $E.$ 
\end{enumerate}
\end{Cor}

\begin{Rem}\rm
Let $\pi:X\to\bP^2$ be a normal triple cover with $\deg\overline{\gD}_{\pi}=6.$ 
\begin{enumerate}
\item[(i)] If $\deg T_{\pi}=3,$ then $\pi$ is a cyclic triple cover since $\bP^2$ is simply connected. 
Conversely, for a reduced cubic curve $\gC\subset\bP^2,$ 
there is a cyclic triple cover whose branch divisor is $2\,\gC.$ 

\item[(ii)] 
In the cases where $\deg T_{\pi}$ is $2$ and $1,$ 
Tokunaga determined the types of $\gD_{\pi}$ in \cite{tokunaga1} and \cite{tokunaga3}, respectively. 

\item[(iii)]
If a reduce sextic curve $\gD$ is defined as (1) in the above corollary, 
then the pair $(G_2,G_3)$ satisfies (2) and (3) in the above corollary. 
In this case, $\gD$ is called a \textit{$(2,3)$-torus sextic} (see \cite{kulikov}). 
Such curves are studied by Oka (\cite{oka1,oka2,oka3}). 

\item[(iv)] 
If $\gD$ is a reduced sextic curve with at most simple singularities, 
Ishida and Tokunaga showed that 
$\gD$ is a $(2,3)$-torus sextic (\cite{ishida}). 
Corollary~\ref{cor 1} is a generalization of this result without such assumption. 
\end{enumerate}
\end{Rem}

\section{Preliminary}
In this section, we recall the theory of triple covers based on Miranda's work \cite{miranda} 
and some facts for locally free sheaves of rank $2$ on $\bP^2.$ 

\subsection{Triple covers}
See \cite{fujita}, \cite{miranda} and \cite{tan} for details and proofs. 
Let $Y$ be a non-singular variety for simplicity. 

\subsubsection{}
Let $\pi:X\to Y$ be a triple cover. 
We denote the kernel of the trace map $\pi_{\ast}\calO_X\to\calO_Y$ by $\calT_{\pi},$ 
which is the locally free $\calO_Y$-module of rank two 
called the \textit{Tschirnhausen module} for $\pi:X\to Y.$ 
Then we have $\pi_{\ast}\calO_X\cong\calO_Y\oplus\calT_{\pi}$ (\cite{miranda}). 

\subsubsection{}\label{cover corre}
Given a locally free sheaf $\calE$ of rank two on $Y,$ 
the $\calO_Y$-algebra structures of $\calA=\calO_Y\oplus\calE$ giving 
triple covers with $\calT_{\pi}=\calE$ are in one-to-one correspondence with $\calO_Y$-linear maps $\Phi:S^3\calE\to\det\calE$ (\cite[Theorem~3.6]{miranda}). 

\subsubsection{}\label{cover mult}
We precisely describe the above correspondence. 
We do this locally on $Y.$ 
Hence we assume that $Y$ is affine and $\calE$ is free. 
Let $\{z,w\}$ be a basis of $\calE$ over $\calO_Y.$ 

1) 
Let $\phi:S^2\calE\to\calA$ be the map induced by the multiplication of $\calA.$ 
Then $\phi$ is of the following form: 
\begin{align*}
\phi(z^2)&=2A+a z+b w, \\
\phi(z w)&=-B-d z-a w, \\
\phi(w^2)&=2C+c z+d w, 
\end{align*}
where $a,\ b,\ c$ and $d$ are in $\calO_Y,$ 
and $A=a^2-b d,$ $B=ad-b c$ and $C=d^2-ac.$ 
In particular, $b\ne0$ and $c\ne0$ if $\calA$ is an integral domain. 

2) 
Define $\Phi:S^3\calE\to\det\calE$ by 
$\Phi(z^3)=-b(z\wedge w),$ $\Phi(z^2w)=a(z\wedge w),$ 
$\Phi(z w^2)=-d(z\wedge w)$ and $\Phi(w^3)=c(z\wedge w).$ 
This definition does not depend on the choice of the basis $\{z,w\}$ of $\calE$ 
and gives the correspondence in (\ref{cover corre}). 

\subsubsection{}\label{cover eq}
Let $S(\calE)$ be the symmetric algebra of $\calE$ and $\bV(\calE)=\mathit{Spec}_Y S(\calE).$ 
This is identified with the total space of the dual vector bundle of $\calE.$ 
Then $X=\mathit{Spec}_Y(\calA)$ is embedded in $\bV(\calE)$ as a closed subvariety by the natural surjection $S(\calE)\to\calA.$ 
The local description of $X$ over $Y$ is as follows: 

Let $z,\ w,\ a,\ b,\ c,\ d,\ A,\ B$ and $C$ be as in (\ref{cover mult}). 
Then $z,\ w$ are fiber coordinate of $\bV(\calE)\cong\bA_Y^2,$ 
and $X$ is defined by 
\[ z^2-\phi(z^2)=z w-\phi(z w)=w^2-\phi(w^2)=0, \]
where $\phi$'s are the polynomials as in (\ref{cover mult}). 
Moreover, $X$ is Cohen-Macaulay. 

\subsubsection{}\label{cover branch}
Assume that $\Phi:S^3\calE\to\det\calE$ gives a normal triple cover $\pi:X\to Y$ with $\calT_{\pi}=\calE$ as above. 
Then the branch divisor $\overline{\gD}_{\pi}$ is locally given by 
\[ D:=B^2-4AC=0, \]
where $A,\ B$ and $C$ are as in (\ref{cover mult}) (\cite[Lemma~4.5]{miranda} and \cite[Theorem~1.3]{tan}). 
Moreover, the line bundle associated to $\overline{\gD}_{\pi}$ is $(\det\calT_{\pi})^{-2}$ (\cite[Proposition~4.7]{miranda}). 

\subsubsection{}\label{cover split}
Let $\pi:X\to Y$ be a normal triple cover. 
If 
$\calT_{\pi}\cong\calL^{-1}\oplus\calM^{-1},$ 
where $\calL$ and $\calM$ are line bundles on $Y,$ 
then $a\in H^0(\calL),$ $b\in H^0(\calL^{2}\otimes\calM^{-1}),$ 
$c\in H^0(\calL^{-1}\otimes\calM^2)$ and $d\in H^0(\calM).$ 
Hence $\calL^{2}\geq\calM$ and $\calM^{2}\geq\calL$ (\cite[Section~6]{miranda}).

\subsubsection{}\label{cover trisec}
Let $\pi:X\to Y$ be a triple cover. 
Suppose that $\pi$ is not \'etale. 
Then $X$ is a triple section in the total space of a line bundle $\calL$ over $Y$ 
if and only if $\calT_{\pi}\cong\calL^{-1}\oplus\calL^{-2}$ (\cite{fujita}).

\subsection{Locally free sheaves of rank $2$ on $\bP^2$}

By a result of Grothendieck, 
each locally free sheaf of rank $2$ on the projective line $\bP^1$ 
is isomorphic to a direct sum $\calO_{\bP^1}(k_1)\oplus\calO_{\bP^1}(k_2),$ 
where integers $k_1,k_2$ are determined up to a permutation. 

Let $\calE$ be a locally free sheaf of rank $2$ on $\bP^2,$ 
and we denote the restriction of $\calE$ to a line $L$ on $\bP^2$ by $\calE_L.$ 
Then $\calE_L$ splits $\calE_L\cong\calO_L(k_{1,L})\oplus\calO_L(k_{2,L})$ as above.  
We put $d(\calE_L)=|k_{1,L}-k_{2,L}|$ for a line $L$ and $d(\calE)=\min\{ d(L) \mid \mbox{$L$ is a line on $\bP^2$} \}.$ 
It is a consequence of the semi-continuity theorems for proper flat morphisms 
that the set of lines $L$ with $d(\calE_L)=d(\calE)$ forms a Zariski-open set in the dual space $\dP^2$ of $\bP^2.$ 
If $d(\calE_L)>d(\calE)$ for a line $L,$ it is called a \textit{jumping line} of $\calE.$ 
If $\calE$ has no jumping lines, $\calE$ is said to be \textit{uniform}.

\section{Proofs}

In this section, 
let $(x_0:x_1:x_2)$ be a system of homogeneous coordinates of $\bP^2,$ 
$U\subset\bP^2$ the open set given by $x_0\ne0,$ 
and put $u_1=x_1/x_0$ and $u_2=x_2/x_0.$ 
We first show the following lemma. 

\begin{Lem}\label{lem. connected}
Let $\pi:X\to\bP^2$ be a normal cover. 
Then $\pi^{\ast}L$ is connected for any line $L\subset\bP^2.$ 
\end{Lem}
\begin{proof}
Let $X_F$ be the fiber product of $X$ and $F$ over $\bP^2,$ 
and $p_X:X_F\to X$ and $\pi_F:X_F\to F$ the projections. 
Note that $\pi^{\ast}L\cong(q\circ \pi_F)^{\ast}L$ for any line $L\subset\bP^2.$ 
Here we regard lines on $\bP^2$ as points of $\dP^2.$ 
By the Stein factorization of $q\circ\pi_F,$ 
we have a finite morphism $\pi':X'\to\dP^2$ and a projective morphism $q':X_F\to X'$ with connected fiber such that $q\circ\pi_F=\pi'\circ q'.$ 
\[ 
\begin{diagram}
\node{X} \arrow{s,l}{\pi} \node{X_F} \arrow{w,t}{p_X} \arrow{s,l}{\pi_F} \arrow{e,t}{q'} \node{X'} \arrow{s,r}{\pi'} \\
\node{\bP^2} \node{F} \arrow{w,b}{p} \arrow{e,b}{q} \node{\dP^2}
\end{diagram}
 \]
If there is a line $L_0$ such that $\pi^{\ast}L$ is disconnected, 
then $\deg\pi'>1,$ 
thus $\pi^{\ast}L$ is disconnected for a general $L\in\dP^2,$ 
which is a contradiction to $X$ irreducible. 
\end{proof}

Let $\pi:X\to \bP^2$ be a normal triple cover with $\deg\overline{\gD}_{\pi}=6.$ 
Then $\det\calT_{\pi}\cong\calO_{\bP^2}(-3)$ by (\ref{cover branch}). 
For a general line $L$ on $\bP^2,$ 
the restriction $\calT_{\pi,L}$ of $\calT_{\pi}$ to $L$ is isomorphic to $\calO_L(-1)\oplus\calO_L(-2)$ by (\ref{cover split}) 
since $\pi|_{\pi^{\ast}L}:\pi^{\ast}L\to L$ is a normal triple cover whose branch divisor is degree $6.$ 
We show that $\calT_{\pi}$ is uniform. 

\begin{Prop}
Let $\pi:X\to\bP^2$ be a normal triple cover with $\deg\overline{\gD}_{\pi}=6.$ 
Then $\calT_{\pi}$ is uniform. 
\end{Prop}
\begin{proof}
Suppose that $\calT_{\pi}$ has a jumping line $L.$ 
Then $\calT_{\pi,L}\cong\calO_{L}(-m-2)\oplus\calO_{L}(m-1)$ for some integer $m>0.$ 
We may assume that $L$ is defined by $x_1=0.$ 
Let $\Phi:S^3\calT_{\pi}\to\det\calT_{\pi}$ be the map corresponding to $\pi:X\to \bP^2.$ 
The restriction of $\Phi$ to $L$ gives sections $a_L,\ b_L,\ c_L$ and $d_L$ of $\calO_{L}(m+2),$ $\calO_{L}(3m+3),$ $\calO_{L}(-3m)$ and $\calO_{L}(1-m),$ respectively. 
In particular, $c_L=0$ and $d_L$ is constant. 
We may assume that $a_L$ (resp. $b_L$) vanishes at $m+2$ (resp. $3m+3$) points of $L\cap U$ if $a_L\ne 0$ (resp. $b_L\ne0$). 
By choosing a basis $\{z,w\}$ of $\calE$ on $U,$ 
$\Phi$ is described as in (\ref{cover mult}) such that 
the restrictions of $a,\ b,\ c$ and $d$ to $L\cap U$ are 
$a_L,\ b_L,\ c_L$ and $d_L,$ respectively.

Suppose $d_L\ne 0.$ 
Since $c_L=0,$ $\pi^{\ast}L$ is locally defined  by 
\begin{align*}
z_L^2-a_L z_L-b_L w_L-2(a_L^2-b_L d_L)&=0, \\
z_L w_L+d_L z_L+a_L w_L+a_L d_L&=0, \\
(w_L-2d_L)(w_L+d_L)&=0, 
\end{align*}
where $z_L$ and $w_L$ are the restrictions of $z$ and $w$ to $L,$ respectively. 
Hence $\pi^{\ast}L$ is disconnected, which is contradiction to Lemma~\ref{lem. connected}. 
Thus $d_L=0.$

Since $c_L=d_L=0,$ $c$ and $d$ have $u_1$ as their factor, say $c=u_1c_1$ and $d=u_1d_1.$ 
Then we have $D=u_1D_1,$ where
\[ D_1=u_1(ad_1-b c_1)^2-4(a^2-u_1b d_1)(u_1d_1^2-ac_1). \]
Since $D=0$ defines a divisor of degree $6$ on $\bP^2,$ 
$D_1=0$ defines one of degree $5.$
Therefore $a^3c_1$ vanishes along $L$ 
since $a_L\in H^0(\calO_L(m+2)).$

Suppose $a_L\ne 0$ on $L,$ then $c_1$ has $u_1$ as its factor, say $c_1=u_1c_2.$ 
We have $D=u_1^2D_2,$ where 
\[ D_2=(ad_1-u_1b c_2)^2-4(a^2-u_1b d_1)(d_1^2-ac_2) \]
Hence $a^2d_1^2-4a^2(d_1^2-ac_2)$ vanishes along $L$ 
since $D_2=0$ defines a quartic curve on $\bP^2$ and  $a_L\in H^0(\calO_L(m+2)).$ 
Then $D$ has $u_1^3$ as its factor, which is a contradiction to (\ref{cover branch}). 
Thus $a_L=0$ on $L,$ and $a$ has $u_1$ as its factor, say $a=u_1a_1.$ 
We have $D=u_1^2D_3,$ where 
\[ D_3=(u_1a_1d_1-b c_1)^2-4u_1(u_1a_1^2-b d_1)(d_1^2-a_1c_1). \]
As above argument, we can see that $b^2c_1^2$ vanishes along $L.$ 
Therefore $D$ has $u_1^3$ as its factor, which is a contradiction, 
and $\calT_{\pi}$ has no jumping lines. 
\end{proof}

From Theorem of \cite{van de ven} and (\ref{cover branch}), we have the following corollary. 

\begin{Cor}\label{cor 2cases}
For a normal triple cover $\pi:X\to\bP^2$ with $\deg\overline{\gD}_{\pi}=6,$ 
$\calT_{\pi}$ is either 
$\calO_{\bP^2}(-2)\oplus\calO_{\bP^2}(-1)$ or $\gO_{\bP^2},$ 
where $\gO_{\bP^2}$ is the cotangent sheaf of $\bP^2.$ 
\end{Cor}

We first consider the case where  $\calT_{\pi}\cong\calO_{\bP^2}(-2)\oplus\calO_{\bP^2}(-1).$ 

\begin{Prop}\label{prop cubic}
Let $\pi:X\to \bP^2$ be a normal triple cover with 
$\calT_{\pi}\cong\calO_{\bP^2}(-2)\oplus\calO_{\bP^2}(-1).$ 
Then $X$ is a normal cubic surface in $\bP^3,$ 
and $\pi$ is identified with a projection centered at a point of $\bP^3\setminus X.$ 
\end{Prop}
\begin{proof}
Since $\bP^2$ is simply connected, $\pi$ is not \'etale. 
By (\ref{cover trisec}), $X$ is a triple section of the total space of the line bundle $\calO_{\bP^2}(1).$ 
Hence it is naturally embedded in the $\bP^1$-bundle $\bP(\calO_{\bP^2}\oplus\calO_{\bP^2}(-1)).$ 
Note that $\bP(\calO_{\bP^2}\oplus\calO_{\bP^2}(-1))$ is a blowing-up of $\bP^3$ at a point, 
and $X$ has no intersections with the exceptional set of the blowing-up under the natural embedding. 
Thus $X$ is a cubic surface in $\bP^3,$ and 
$\pi$ is identified with a projection centered at a point of $\bP^3.$ 
\end{proof}

Next we show that $X$ is a $\bP^1$-bundle over an elliptic curve 
if $\calT_{\pi}\cong\gO_{\bP^2}.$ 
Let $V$ be a vector space of dimension $3,$ 
and $v_0,v_1,v_2$ a basis of $V.$ 
We regard $\bP^2$ as the set of $1$-dimensional subspaces of $V,$ $\bP(V).$ 
Then $\dP^2=\bP(V^{\ast}),$ where $V^{\ast}$ is the dual space of $V,$ 
and we can regard $x_0,\ x_1$ and $x_2$ as the dual of $v_0,\ v_1$ and $v_2,$ respectively.  
Note that $V$ and $V^{\ast}$ are naturally identified with
$H^0(\dP^2,\calO_{\dP^2}(1))$ and $H^0(\bP^2,\calO_{\bP^2}(1)),$ respectively.

\begin{Lem}\label{lem. section}
There is a natural isomorphism of vector spaces
\[ \gh : H^0(\dP^2,\calO_{\dP^2}(3))\overset{\sim}{\to} H^0(\bP^2,(S^3\gO_{\bP^2})^{\ast}\otimes\det\gO_{\bP^2}). \]
Furthermore $\gh$ is defined as follows:

Let $f$ be a global section of $\calO_{\dP^2}(3)$ as follows: 
\begin{align*} 
f=&\  t_{{1}}{v_{{0}}}^{3}+3\,t_{{2}}{v_{{0}}}^{2}v_{{1}}+3\,t_{{3}}{v_{{0}}
}^{2}v_{{2}}+3\,t_{{4}}v_{{0}}{v_{{1}}}^{2}+3\,t_{{5}}v_{{0}}v_{{1}}v_
{{2}} \\ 
&+3\,t_{{6}}v_{{0}}{v_{{2}}}^{2}+t_{{7}}{v_{{1}}}^{3}+3\,t_{{8}}{v
_{{1}}}^{2}v_{{2}}+3\,t_{{9}}v_{{1}}{v_{{2}}}^{2}+t_{{10}}{v_{{2}}}^{3
},
 \end{align*}
where $t_1,\dots,t_{10}\in\bC.$ 
Then  $\gh(f)$
is locally 
\[ -b_f\, (z^3)^{\ast}\otimes(z\wedge w)+a_f\, (z^2w)^{\ast}\otimes(z\wedge w) 
-d_f\, (z w^2)^{\ast}\otimes(z\wedge w) + c_f\, (w^3)^{\ast}\otimes(z\wedge w), \] 
where
\begin{align*}
a_f&= -\,t_{{1}}u_{{2}}{u_{{1}}}^{2}+2\,t_{{2}}u_{{2}}u_{{1}}
+\,t_{{3}}{u_{{1}}}^{2}-\,t_{{4}}u_{{2}}-\,t_{{5}}u_{{1}}+\,t_{{8}}, \\
b_f&= t_{{1}}{u_{{1}}}^{3}-3\,t_{{2}}{u_{{1}}}^{2}+3\,t_{{4}
}u_{{1}}-t_{{7}},  \\
c_f&=-t_{{1}}{u_{{2}}}^{3}+3\,t_{{3}}{u_{{2}}}^{2}-3\,t_{{6}}u
_{{2}}+t_{{10}}, \\
d_f&= t_{{1}}{u_{{2}}}^{2}u_{{1}}-\,t_{{2}}{u_{{2}}}^{2}
-2\,t_{{3}}u_{{2}}u_{{1}}+\,t_{{5}}u_{{2}}+\,t_{{6}}u_{{1}}-\,t_{{9}}, 
\end{align*}
and $z$ and $w$ are the differential forms $d u_1$ and $d u_2,$ respectively. 
\end{Lem}
\begin{proof}
Note that the $\bP^1$-bundle $\bP(\gO_{\bP^2})$ over $\bP^2$ is isomorphic to $F,$ 
and the projection $\bP(\gO_{\bP^2})\to\bP^2$ coincides with $p:F\to\bP^2.$ 
The canonical embedding $F\hookrightarrow\bP^2\times\dP^2$ is given by the surjection $\ga$ of the exact sequence 
\[ 0 \to \calO_{\bP^2}(-1) \to V\otimes\calO_{\bP^2} \overset{\ga\ }{\to} \gO_{\bP^2}^{\ast}(-1) \to 0, \]
where $\ga$ is locally defined by $\ga(v_0)=-u_1z^{\ast}/x_0-u_2w^{\ast}/x_0,$ $\ga(v_1)=z^{\ast}/x_0$ and $\ga(v_2)=w^{\ast}/x_0$ (cf. \cite{hartshorne}). 
Let $\calO_{F}(1)$ be an invertible sheaf on $F$ such that $p_{\ast}\calO_F(1)\cong\gO_{\bP^2}^{\ast}(-1).$ 
Then $\ga$ induces an isomorphism $q^{\ast}\calO_{\dP^2}(3)\overset{\sim}{\to}\calO_F(3).$ 
In particular, since $H^0(F,\calO_F(3))$ and $H^0(F,q^{\ast}\calO_{\dP^2}(3))$ are identified with 
$H^0(\bP^2,S^3(\gO_{\bP^2}^{\ast}(-1)))$ and $S^3V,$ respectively, 
the symmetric product of $\ga$ gives an isomorphism 
\[ S^3\ga:H^0(\dP^2,\calO_{\dP}(3))\overset{\sim}{\to}H^0(\bP^2,S^3(\gO_{\bP^2}^{\ast}(-1))). \]
Note that there is the natural isomorphism $\gk:(S^3\gO_{\bP^2})^{\ast}\otimes\det\gO_{\bP^2}\overset{\sim}{\to} S^3(\gO_{\bP^2}^{\ast}(-1)),$
which is locally defined by 
\begin{align*}
&(z^3)^{\ast}\otimes(z\wedge w) \mapsto (z^{\ast})^3/x_0^3,
&(z^2w)^{\ast}\otimes(z\wedge w) \mapsto 3\,(z^{\ast})^2w^{\ast}/x_0^3, \\
&(w^3)^{\ast}\otimes(z\wedge w) \mapsto (w^{\ast})^3/x_0^3, 
&(z w^2)^{\ast}\otimes(z\wedge w) \mapsto 3\,z^{\ast}(w^{\ast})^2/x_0^3.
\end{align*}
Therefore we obtain a natural isomorphism $\gh=\gk^{-1}\circ S^3\ga.$ 
We can see the second assertion by direct computation. 
\end{proof}

Let $f=f(v_0,v_1,v_2)$ be a global section of $\calO_{\dP^2}(3)$ as in Lemma~\ref{lem. section}. 
Put $\gd_f=\gd_f(u_1,u_2)$ the discriminant of $f(-u_1v_1-u_2v_2,v_1,v_2)$ with respect to $v_1$ and $v_2.$ 
We denote $D$ for $\gh(f)$ in (\ref{cover branch}) by $D_f.$ 

\begin{Lem}\label{lem discrim}
Let $f,\ \gd_f$ and $D_f$ be as above. 
Then $\gd_f$ coincides with $D_f$ up to multiplying constants. 
\end{Lem}
\begin{proof}
We can check the assertion by direct computation. 
\end{proof}

We can identify 
$H^0(\bP^2,(S^3\gO_{\bP^2})^{\ast}\otimes\det\gO_{\bP^2})$ with 
$H^0(\dP^2,\calO_{\dP^2}(3))$ by the isomorphism $\gh.$

\begin{Prop}\label{prop diff-branch}
Let $\pi:X\to \bP^2$ and $\pi':X'\to\bP^2$ be normal triple covers corresponding to $f,f'\in H^0(\dP^2,\calO_{\dP^2}(3)),$ respectively. 
Then $\gD_{\pi}=\gD_{\pi'}$ if and only if 
$f'=\gl f$ for some non-zero constant $\gl.$ 
In particular, there is an isomorphism $\gs:X\to X'$ such that $\pi=\pi'\circ\gs$ in this case. 
\end{Prop}
\begin{proof}
We first show that $f$ is irreducible. 
Assume that $f$ is reducible. 
We may assume that $f$ has $v_2$ as its factor (i.e. $t_i=0$ for $i=7,\dots,10$), and $b_f\ne0.$ 
Then $z$ satisfies the following equation (cf. \cite{miranda}): 
\[ z^3-3\, A_f z + (b_f B_f-2\,a_f A_f)=0, \]
where $A_f=a_f^2-b_f d_f$ and $B_f=a_f d_f-b_f c_f.$ 
The above polynomial is divided by $z-t_2u_1u_2+t_3u_1^2+2\,t_4u_2-t_5u_1.$ 
Thus $X$ is reducible, which is a contradiction. 

Hence we may assume that $f$ is irreducible. 
Let $\gC^{\vee}$ be the curve on $\dP^2$ defined by $f=0.$ 
Then $\gd_f=0$ defines a divisor of degree $6$ on $\bP^2$ whose support is 
the union of the dual curve of $\gC^{\vee}$ and the lines corresponding to the singular points of $\gC^{\vee}.$ 
Therefore, by Lemma~\ref{lem discrim}, $\gD_{\pi}=\gD_{\pi'}$ if and only if $f'=\gl f.$ 
\end{proof}

The above proposition enables us to distinguish normal covers for $\gO_{\bP^2}$ by their branch loci. 

\begin{Prop}\label{prop flag}
Let $\pi:X\to\bP^2$ be a triple cover for $f\in H^0(\dP^2,\calO_{\dP^2}(3)).$ 
Then $X$ is normal 
if and only if 
the curve $\gC^{\vee}\subset\dP^2$ defined by $f=0$ is smooth. 
Moreover, if $X$ is normal, then 
$\gD_{\pi}$ is the dual curve of $\gC^{\vee},$ 
$X$ is isomorphic to $q^{-1}(\gC^{\vee})\subset F,$ and
$\pi$ is identified with the restriction of $p:F\to\bP^2$ to $q^{-1}(\gC^{\vee}).$ 
\end{Prop}

\begin{proof}
Suppose $X$ is normal. 
By the proof of Proposition~\ref{prop diff-branch}, 
$\gC^{\vee}$ is reduced and irreducible. 
Put $X'=q^{-1}(\gC^{\vee}).$ 
Then $p|_{X'}:X'\to\bP^2$ is a triple cover with $\calT_{p|_{X'}}\cong\gO_{\bP^2}$ (\cite[Proposition~8.1]{miranda}) whose branch locus is given by $\gd_f=0.$ 
Thus there is an isomorphism $\gs:X\to X'$ such that $\pi=p|_{X'}\circ\gs$ by Proposition~\ref{prop diff-branch}. 
If $\gC^{\vee}$ is singular, then $X'$ is not normal. 
Therefore $\gC^{\vee}$ is smooth. 
Conversely, 
we can see that $X$ is normal if $\gC^{\vee}$ is smooth as above argument. 
\end{proof}

\begin{proof}[Proof of Theorem~\ref{th main}]
Let $\pi:X\to\bP^2$ be a normal triple cover with $\deg\overline{\gD}_{\pi}=6.$ 
Then $X$ is either a normal cubic surface in $\bP^3$ or 
$q^{-1}(\gC^{\vee})\subset F$ for some smooth cubic curve 
$\gC^{\vee}\subset\dP^2$ by Corollary~\ref{cor 2cases}, Proposition~\ref{prop cubic} and \ref{prop flag}. 
We can check easily that 
the $9$ cusps of $\gD_{\pi}$ are total branched points of $\pi$ 
if $X=q^{-1}(\gC^{\vee})\subset F.$

Assume that $\gD_{\pi}$ is a sextic curve with $9$ cusps, and
the $9$ cusps are total branched points of $\pi.$ 
Suppose that $X$ is a normal cubic surface. 
Then we may assume that there are homogeneous polynomials $G_i(x_0,x_1,x_2)$ of degree $i$ for $i=2,3$ 
such that $X$ is defined by $x_3^3+3\, G_2\, x_3+2\, G_3=0,$ 
and $\pi$ is the projection centered at $P=(0:0:0:1).$ 
Here we regard $(x_0:x_1:x_2:x_3)$ as a system of homogeneous coordinates of $\bP^3.$ 
In this case, $\overline{\gD}_{\pi}=\gD_{\pi}$ is defined by $G_2^3+G_3^2=0.$ 
Since $\gD_{\pi}$ has just $9$ cusps as its singularities, 
$G_2=0$ and $G_3=0$ define reduced curves $\gC_2$ and $\gC_3,$ respectively, 
such that they intersect transversally each other. 
Then it is easy to see that the total branched points of $\pi$ are just $6$ intersection points of $\gC_2$ and $\gC_3,$ which is a contradiction. 
Hence $X$ is a subvariety of $F.$ 
\end{proof}

Our proof of Corollary~\ref{cor 1} is the same as the proof of \cite[Theorem~0.1]{ishida}, 
and we omit it. 



\end{document}